\documentclass[reqno]{amsart}
\usepackage{amssymb}
\usepackage{relsize}
\usepackage{enumerate}
\usepackage{thmtools,booktabs,etoolbox,extarrows, mathtools, mathrsfs}
\usepackage[all,cmtip]{xy}
\allowdisplaybreaks
\usepackage[backref=page]{hyperref}

\makeatletter
\patchcmd{\@setaddresses}{\nobreak}{\allowbreak}{}{}
\patchcmd{\@setaddresses}{\nobreak\addvspace}{\allowbreak\addvspace}{}{}
\patchcmd{\@setaddresses}{}{}{}{}
\patchcmd{\@setaddresses}{\bigskipamount}{\smallskipamount}{}{}
\patchcmd{\@setaddresses}{\bigskipamount}{\smallskipamount}{}{}
\makeatother

\usepackage{comment} 

\let\oldtocsection=\tocsection 
\let\oldtocsubsection=\tocsubsection 
\renewcommand{\tocsection}[2]{\hspace{0mm}\oldtocsection{#1}{#2}}
\renewcommand{\tocsubsection}[2]{\hspace{2em}\oldtocsubsection{#1}{#2}}

\usepackage{hyperref}
\usepackage{cleveref}
\hypersetup{
    colorlinks=true,
    linkcolor=blue,
    filecolor=magenta,      
    urlcolor=cyan,
}

\numberwithin{equation}{section}
\newtheorem{theorem}{Theorem}[section]
\newtheorem*{theorem*}{Theorem}
\newtheorem{corollary}[theorem]{Corollary}
\newtheorem*{corollary*}{Corollary}
\newtheorem*{conjecture*}{Conjecture}
\newtheorem{lemma}[theorem]{Lemma}
\newtheorem{proposition}[theorem]{Proposition}
\theoremstyle{definition}

\newtheorem{definition}[theorem]{Definition}
\newtheorem{remark}[theorem]{Remark}
\newtheorem{example}[theorem]{Example}

\def\TT{\mathcal{T}}

\newcommand{\smodcat}{\underline{\operatorname{mod}}}
\def\Id{\operatorname{Id}}

\def\End{\operatorname{End}}
\def\Hom{\operatorname{Hom}}

\def\Im{\operatorname{Im}}

\def\Cone{\operatorname{Cone}}

\DeclareMathOperator{\Coker}{Coker}

\newcommand{\authorstar}{\textsuperscript{*}}

\begin{document}

\makeatletter
\renewcommand{\author}[2][]{%
  \ifx\@empty\authors
    \gdef\authors{#2}%
  \else
    \g@addto@macro\authors{\and#2}%
  \fi
  \def\@tempa{#1}%
  \ifx\@empty\shortauthors
    \ifx\@empty\@tempa
      \gdef\shortauthors{#2}%
    \else
      \gdef\shortauthors{#1}%
    \fi
  \else
    \ifx\@empty\@tempa
      \g@addto@macro\shortauthors{\and#2}%
    \else
      \g@addto@macro\shortauthors{\and#1}%
    \fi
  \fi
  \g@addto@macro\addresses{%
    \par\addvspace\bigskipamount
    \noindent{\let\authorstar\@empty\bfseries #2}%
    \par\nobreak
  }%
}
\makeatother

\title{Pretriangulated categories failing (TR4)}

\author[X.-W. Chen]{Xiao-Wu Chen}
\address{
School of Mathematical Sciences,
University of Science and Technology of China,
Hefei 230026, Anhui Province, P.~R.~China.}
\email{xwchen@mail.ustc.edu.cn}

\author[S. O. Ivanov]{Sergei O. Ivanov}
\address{
Beijing Key Laboratory of Topological Statistics and Applications
for Complex Systems,
Beijing Institute of Mathematical Sciences and Applications (BIMSA),
Beijing 101408, P.~R.~China.}
\email{ivanov.s.o.1986@gmail.com, ivanov.s.o.1986@bimsa.cn}

\author[J. Liu]{Jian Liu}
\address{
School of Mathematics and Statistics,
and Hubei Key Laboratory of Mathematical Sciences,
Central China Normal University,
Wuhan 430079, P.~R.~China.}
\email{jianliu@ccnu.edu.cn}

\author[X.-S. Lu]{Xue-Song Lu\authorstar}
\address{
School of Mathematical Sciences,
Shanghai Jiao Tong University,
Shanghai 200240, P.~R.~China.}
\email{leocedar@sjtu.edu.cn}

\author[V. Sosnilo]{Vladimir Sosnilo}
\address{
RIKEN iTHEMS,
Saitama 351-0198, Japan}
\email{vsosnilo@gmail.com, vladimir.sosnilo@riken.jp}

\author[C. Zhang]{Chencheng Zhang}
\address{
School of Mathematical Sciences,
Shanghai Jiao Tong University,
Shanghai 200240, P.~R.~China.}
\email{zhangchencheng@sjtu.edu.cn}

\dedicatory{Dedicated to Professor Alexander Generalov and Professor Pu Zhang}

\makeatletter
\@namedef{subjclassname@2020}{\textup{2020} Mathematics Subject Classification}
\makeatother

\thanks{$^*$ The corresponding author}
\date{\today}

\thanks{This work is  supported by 
National Key R\textup{\&}D Program of China (No. 2024YFA1013801), the National Natural Science Foundation of China (Nos. 12401046, 12325101, and 12131015), and the JSPS KAKENHI Grant (No. JP26K16973).}

\begin{abstract}
We introduce gemstones in a pretriangulated category, which are used to perturb pretriangulated structures. For a triangulated category $\TT$ whose stable category of coherent functors admits a gemstone, we construct pretriangulated structures on $\TT$ which fail the octahedral axiom (TR4).  This confirms a long-standing conjecture on triangulated categories.
\end{abstract}

\maketitle

\setcounter{tocdepth}{1}

\section{Introduction}
\label{sec:introduction}

Triangulated categories were introduced in \cite{Ver96}; compare \cite{Pup67}. Since then, they have become indispensable in many areas of mathematics, including algebraic geometry, algebraic topology, mathematical physics and representation theory.

Let $\mathcal{T}$ be an additive category equipped with an autoequivalence $\Sigma$. By a \emph{pretriangulated structure} on $(\mathcal{T}, \Sigma)$, we mean a chosen class $\Delta$ of triangles which satisfies (TR1)--(TR3). If in addition it satisfies the octahedral axiom (TR4), it is called a \emph{triangulated structure}. In this situation, the triple $(\mathcal{T}, \Sigma, \Delta)$ is called a triangulated category in the sense of \cite{Ver96}.

Among the axioms, the octahedral axiom (TR4) seems to be most subtle.  It plays a fundamental role in the localization theory \cite{Ver96} of triangulated categories, $t$-structures \cite{BBD82}, stability conditions \cite{Bri07}, dimensions \cite{Rou08} of triangulated categories, and approximable triangulated categories \cite{Nee24}. 

A long-standing  question on triangulated categories is whether (TR4) is independent of (TR1)--(TR3). This is precisely the content of the following well-known conjecture, formulated in \cite[Conjecture~9.9]{Bel00} and attributed to Keller and Neeman; compare \cite[Remark~1.3.15]{Nee01}.

\vskip 5pt

\noindent {\bf Conjecture}.  \emph{The octahedral axiom {\rm (TR4)} is not redundant. In other words, there exists a pretriangulated structure which does not satisfy {\rm (TR4)}.}

\vskip 5pt

The main contribution of this work is the new concept \emph{gemstone} in a pretriangulated category; see Definition~\ref{defn:gemstone}. Roughly speaking, a gemstone is a rigid $2$-periodic object which admits an almost-vanishing endomorphism. Such an object yields morphisms that are not lightning flashes \cite{Nee91, CF22} on certain exact triangles; see Proposition~\ref{prop:gem3}. 

The following main result relates gemstones to the conjecture.

\begin{theorem}\label{main-result}
  Let $k$ be a field and $(\mathcal{T}, \Sigma, \Delta)$ be a Hom-finite $k$-linear triangulated category with split idempotents. Assume that the stable category of coherent functors on $\mathcal{T}$ admits a gemstone. Then there exists a pretriangulated structure on $(\mathcal{T}, \Sigma)$ failing the octahedral axiom {\rm (TR4)}. 
\end{theorem}

We mention that the study of coherent functors on triangulated categories goes back to \cite{Freyd, Ver96}. The category of coherent functors on a triangulated category is a Frobenius abelian category, and by \cite{Hap88} its stable category becomes a triangulated category. 

The main tool in the proof of \Cref{main-result} is the description of pretriangulated structures by Heller parameters \cite{Hel68}, which enables us to produce new pretriangulated structures by using the central action to \emph{perturb} Heller parameters. Another key ingredient is the equivalent form of (TR4) in \cite{Nee91}, that  involves cones of morphisms between exact triangles; see also \cite{Nee01}. We observe an obstruction on lightning flashes which arises from two triangulated Heller parameters; see Proposition~\ref{prop:lightning-flash-obstruction}.

\Cref{main-result}  allows us to give concrete examples to confirm the conjecture.  Let $R$ be a finite-dimensional commutative local symmetric $k$-algebra. Take $A=R[x]/(x^{3m})$ to be the truncated polynomial algebra  with $m\geq 2$. 

The stable module category $\smodcat(A)$ of $A$ has a canonical triangulated structure, whose suspension functor is given by the cosyzygy functor $\Omega_A^{-1}$. It turns out that the stable category of coherent functors on $\smodcat(A)$ does admit a gemstone; see Example~\ref{exm:longexact}. Consequently, we have the following concrete result which confirms the conjecture above.

\begin{corollary}\label{application}
Let  $A=R[x]/(x^{3m})$ with $m\geq 2$ be as above. Then there is a pretriangulated structure on $(\smodcat(A),\Omega_A^{-1})$ failing the octahedral axiom {\rm (TR4)}. 
\end{corollary}

It is worth pointing out that, when the ground field $k$ is infinite, the pair $(\smodcat(A),\Omega_A^{-1})$ admits infinitely many such pretriangulated structures; see \Cref{rem:mainthm}.

The paper is organized as follows. In \Cref{tri-pre}, we recall lightning flashes between exact triangles in a pretriangulated category. In \Cref{Hel-para}, we recall Heller parameters and their role in describing pretriangulated structures. In \Cref{sec:lightning-flash}, we obtain an obstruction on lightning flashes for two given triangulated Heller parameters.  In \Cref{sec:gemstone}, we introduce gemstones in a pretriangulated category and show that they give rise to this obstruction. We emphasize that Section 5 relies on the Auslander--Reiten theory \cite{Hap88} for pretriangulated categories. In \Cref{sec:main-results}, we prove our main result, \Cref{main-result}. Finally, in \Cref{sec:example}, we give a method for constructing gemstones through Calabi-Yau categories \cite{Kel08} and present the explicit examples above.

We assume that all categories under consideration are skeletally small. The suspension functor on any pretriangulated category will be denoted by $\Sigma$.

\section{Pretriangulated structures and lightning flashes}\label{tri-pre}

In this section, we recall basic facts on pretriangulated categories and lightning flashes.

Let $\TT$ be an additive category, and $\Sigma \colon \TT\rightarrow\TT$ be an autoequivalence. For the pair $(\TT,\Sigma)$, a triangle means a sequence
\[
\xi \colon\quad X \xlongrightarrow{f}Y \xlongrightarrow{g}Z \xlongrightarrow{h} \Sigma X
\]
such that $g\circ f=0$, $h\circ g=0$ and $ (\Sigma f)\circ h=0$. 
Given another triangle 
\[
\xi'\colon\quad X'\xlongrightarrow{f'}Y'\xlongrightarrow{g'}Z'\xlongrightarrow{h'} \Sigma X',
\]
a morphism $(a,b,c)\colon \xi \rightarrow \xi'$ between these triangles  consists of three morphisms
\[
a\colon X\longrightarrow X',\quad
b\colon Y\longrightarrow Y',\quad
c\colon Z\longrightarrow Z'
\]
such that the following diagram commutes.
\[
\xymatrix@R=2em@C=4em{
X \ar[r]^-f \ar[d]_-a
&Y \ar[r]^-g \ar[d]_-b
&Z \ar[r]^-h \ar[d]_-c
&\Sigma  X \ar[d]^-{\Sigma a}
\\
 X' \ar[r]^-{f'}
&Y' \ar[r]^-{g'}
&Z' \ar[r]^-{h'}
&\Sigma X'
}
\]

Recall that a \emph{pretriangulated structure} on $(\TT, \Sigma)$ is a chosen class $\Delta$ of triangles, which is closed under isomorphisms and satisfies the axioms (TR1)-(TR3). In this situation, we  will say that $(\TT, \Sigma, \Delta)$ is a pretriangulated category, and triangles in $\Delta$ are called exact triangles. 

The pretriangulated category $(\TT, \Sigma, \Delta)$ is \emph{triangulated} provided further that $\Delta$ satisfies the octahedral axiom (TR4). We will recall an equivalent form of (TR4).

For a morphism $(a, b, c)\colon \xi\rightarrow \xi'$ between triangles,  its \emph{truncated cone}  is defined to be the
following triangle $\operatorname{C}(a,b,c).$ 
\[
\quad
Y\oplus X'
\xlongrightarrow{\left(\begin{smallmatrix}
-g&0\\
b&f'
\end{smallmatrix}\right)}
Z\oplus Y'
\xlongrightarrow{
\left(\begin{smallmatrix}
-h&0\\
c&g'
\end{smallmatrix}
\right)
}
\Sigma X\oplus Z'
\xlongrightarrow{
\left(
\begin{smallmatrix}
-\Sigma f& 0\\
\Sigma a& h'
\end{smallmatrix}
\right)
}
\Sigma Y\oplus\Sigma X'
\]

The following fundamental result is due to \cite[Theorem~1.8]{Nee91}; see also \cite[Proposition~1.4.6~and~Remark~1.4.7]{Nee01}.

\begin{theorem} \label{thm:good-fill-in-criterion}
Let $(\TT,\Sigma,\Delta)$ be a pretriangulated category. Then it is triangulated if and only if  for any morphism $(a, b,c)\colon \xi \rightarrow \xi'$ between exact triangles, there exists some morphism $c'\colon Z\rightarrow Z'$ such that $(a, b,c')\colon \xi \rightarrow \xi'$ is a morphism whose truncated cone $C(a, b, c')$ is also exact. 
\end{theorem}

The following concept is due to  \cite[Definition~3.5]{CF22}, which has already appeared implicitly in \cite[Lemmas~2.1~and~2.2]{Nee91}.

\begin{definition}
Let $(\mathcal T, \Sigma, \Delta)$ be a pretriangulated category.  Given two exact triangles $\xi$ and $\xi'$
\[
\xymatrix@R=1.2em@C=2,5em{
\xi\colon\quad X \ar[r]^-u
&Y \ar[r]^-v
&Z \ar[r]^-w \ar[dd]_(.35)\ell
&\Sigma X\\
&&&\\
\xi'\colon\quad X' \ar[r]^-{u'}
&Y' \ar[r]^-{v'}
&Z' \ar[r]^-{w'}
&\Sigma X',
\save
\ar@{.}
  "1,4";
  "2,4"+<0pt,10pt>
\ar@{.}@/^5pt/
  "2,4"+<0pt,10pt>;
  "2,4"+<-10pt,0pt>
\ar@{.}
  "2,4"+<-10pt,0pt>;
  "2,2"+<10pt,0pt>
  _(.25){\varphi}
\ar@{.}@/_5pt/
  "2,2"+<10pt,0pt>;
  "2,2"+<0pt,-10pt>
\ar@{.>}
  "2,2"+<0pt,-10pt>;
  "3,2"
\restore
}
\]
a morphism $\ell\colon Z\rightarrow Z'$ is called a \emph{lightning flash} from $\xi$ to $\xi'$,  if there exists a morphism $\varphi\colon \Sigma X\rightarrow Y'$ satisfying
$$\ell=v'\circ \varphi \circ  w.$$
When $\xi=\xi'$, we will simply say that $\ell \colon Z\rightarrow Z$ is a lightning flash on $\xi$.
\end{definition}

Let $\mathcal{C}$ be an additive category. A complex in $\mathcal{C}$ is usually denoted by $P^\bullet=(P^n, d_P^n)_{n\in \mathbb{Z}}$. Its suspended complex $\Sigma P^\bullet$ is defined by
$$(\Sigma P^\bullet )^n=P^{n+1} \mbox{ and } d_{\Sigma P}^n=-d_P^{n+1}.$$
For a cochain map $f^\bullet=(f^n)_{n\in \mathbb{Z}}\colon P^\bullet \rightarrow Q^\bullet$, its \emph{mapping cone} ${\rm Cone}(f^\bullet)$ is a complex defined by 
$${\rm Cone}(f^\bullet)^n=P^{n+1}\oplus Q^n \mbox{  and } d_{{\rm Cone}(f^\bullet)}^n=\begin{pmatrix} -d_P^{n+1} & 0\\
f^{n+1} & d_Q^n\end{pmatrix}.$$
Here, elements in ${\rm Cone}(f^\bullet)$ are viewed as column vectors. 

Denote by $\mathbf{K}(\mathcal{C})$ the homotopy category of complexes in $\mathcal{C}$. Recall that each cochain map $f^\bullet \colon P^\bullet \rightarrow Q^\bullet$ induces a \emph{standard triangle}.
$$P^\bullet \xrightarrow{f^\bullet} Q^\bullet \xrightarrow{\binom01}  {\rm Cone}(f^\bullet) \xrightarrow{(1,0)} \Sigma P^\bullet $$

Suppose that $f^\bullet, g^\bullet\colon P^\bullet\rightarrow Q^\bullet$ are two cochain maps such that  $f^\bullet$ is homotopic to $g^\bullet$ via homotopy $H=(h^n\colon P^{n+1}\rightarrow Q^{n})_{n\in \mathbb{Z}}$. It is well known that the homotopy induces a cochain isomorphism 
$$\theta^\bullet_H\colon {\rm Cone}(f^\bullet) \longrightarrow {\rm Cone}(g^\bullet) $$
which is  given by 
\begin{align}\label{equ:theta_H}
    \theta_H^n=\begin{pmatrix}
    1 & 0\\
    h^n & 1
\end{pmatrix}.
\end{align} 

Assume that  $f^\bullet$ is homotopic to $g^\bullet$ via another homotopy $H'=((h')^n\colon P^{n+1}\rightarrow Q^{n})_{n\in \mathbb{Z}}$. We have the corresponding cochain isomorphism 
$$\theta^\bullet_{H'}\colon {\rm Cone}(f^\bullet) \longrightarrow {\rm Cone}(g^\bullet).$$

We will use the following concrete lightning flash in Section~\ref{sec:lightning-flash}. 

\begin{proposition}\label{factsoncomplex}
     Keep the assumptions above. Then the cochain map 
     $$(\theta^\bullet_H)^{-1}\circ \theta^\bullet_{H'}-\Id\colon {\rm Cone}(f^\bullet) \longrightarrow {\rm Cone}(f^\bullet) $$ 
     is a lightning flash  on the standard triangle.
\[
\xymatrix@R=2em@C=3em{
P^\bullet \ar[r]^-{f^\bullet}
&Q^\bullet \ar[r]^-{\binom01}
&\Cone(f^\bullet) \ar[d]^-{(\theta^\bullet_H)^{-1}\circ \theta^\bullet_{H'}-\Id}  \ar[r]^-{(1,0)}
&\Sigma P^\bullet
\\
P^\bullet \ar[r]_-{f^\bullet}
&Q^\bullet \ar[r]_-{\binom01}
&\Cone(f^\bullet) \ar[r]_-{(1,0)}
&\Sigma P^\bullet
}
\]
\end{proposition}

\begin{proof}
The difference between $H'$ and $H$ yields a cochain map $\iota^\bullet \colon \Sigma P^\bullet \rightarrow Q^\bullet$, whose $n$-th component is given by $\iota^n={(h')}^n-h^n$. We observe that 
$$(\theta^\bullet_H)^{-1}\circ \theta^\bullet_{H'}-\Id = \binom 01 \circ \iota^\bullet \circ (1,0).$$  
Then the required statement follows immediately.
\end{proof}

Let $(\mathcal{T}, \Sigma, \Delta)$ be a pretriangulated category. Recall from \cite{Kun06} that the \emph{triangle center} $Z_\vartriangle(\mathcal{T})$ of $\mathcal{T}$ consists of natural transformations on the identity triangle endofunctor ${\rm Id}_\mathcal{T}$; see also \cite[Subsection~2.2]{CY18}. In other words, the elements are natural transformations $\lambda \colon {\rm Id}_\mathcal{T}\rightarrow {\rm Id}_\mathcal{T}$ satisfying $\lambda \Sigma=\Sigma \lambda$. It is a commutative ring, whose product is given by the composition of natural transformations. Denote by $Z_\vartriangle(\mathcal{T})^\times$ its unit group, that is, the multiplicative group formed by invertible elements in $Z_\vartriangle(\mathcal{T})$.

\section{Heller parameters}\label{Hel-para}

In this section, we study Heller parameters in \cite{Hel68}, which classify  pretriangulated structures.  Throughout this section, $\mathcal{A}$ is a Frobenius abelian category. Denote by $\mathcal{P}$ the full subcategory formed by projective-injective objects, and by $\underline{\mathcal{A}}$ the stable category.

For each object $X$, we fix an exact sequence
\begin{align}\label{ex:Omega}
0\longrightarrow X \xrightarrow{i_X}  I_X  \xrightarrow{p_X} \Omega^{-1}X\longrightarrow 0
\end{align}
with $I_X$ injective. This gives rise to the cosyzygy endofunctor $\Omega^{-1}$ on $\underline{\mathcal{A}}$. By \cite[Section I.2]{Hap88}, the pair $(\underline{\mathcal{A}}, \Omega^{-1})$ has a canonical triangulated structure, whose exact triangles are induced by short exact sequences in $\mathcal{A}$.

We fix an autoequivalence $\mathbb{S}$ on $\mathcal{A}$. The induced autoequivalence on $\underline{\mathcal{A}}$ is still denoted by $\mathbb{S}$. We have a canonical natural isomorphism 
$$ \sigma\colon \Omega^{-1}\mathbb{S} \stackrel{\sim}{\longrightarrow} \mathbb{S}\Omega^{-1}$$
on $\underline{\mathcal{A}}$. To be more precise, it is induced by the following commutative diagram in $\mathcal{A}$.
\[
\xymatrix@R=1.5em@C=2.5em
{
0\ar[r] & \mathbb{S}X \ar@{=}[d]\ar[r]^-{i_{\mathbb{S}X}} & I_{\mathbb{S}X} \ar@{.>}[d] \ar[r]^-{p_{\mathbb{S}X}} & \Omega^{-1}\mathbb{S}X \ar[r] \ar@{.>}[d]^-{\sigma_X} & 0\\
0\ar[r] & \mathbb{S}X \ar[r]^-{\mathbb{S}(i_X)} & \mathbb{S}(I_{X}) \ar[r]^-{\mathbb{S}(p_X)} & \mathbb{S}\Omega^{-1}X \ar[r] & 0
}
\]

The following consideration goes back to \cite[Section~16]{Hel68}. 

\begin{definition}\label{def:heller-parameter}
A \emph{Heller parameter} on the pair $(\mathcal{A};\mathbb S)$ is a
natural isomorphism
\[
\delta\colon \mathbb S\xlongrightarrow{\sim}\Omega^{-3}
\]
on $\underline{\mathcal A}$ satisfying
$
\Omega^{-1}\delta+(\delta\Omega^{-1})\circ\sigma=0.
$
\end{definition}

\begin{remark}
    The equation in the definition above means that $\delta$ is a natural transformation between the two \emph{triangle} endofunctors $(\mathbb{S}, \sigma^{-1})$ and $(\Omega^{-3}, -\Id)$ on $\underline{\mathcal{A}}$. Here, $-\Id$ means the negative of the identity transformation on $\Omega^{-4}$.  
\end{remark}

Let $\delta$ be  a Heller parameter on $(\mathcal{A}; \mathbb S)$. We  denote by $\Delta_\delta$ the class of triangles 
$$P^0\xlongrightarrow{f} P^1\xlongrightarrow{g} P^2\xlongrightarrow{h} \mathbb SP^0$$ 
with all $P^i\in \mathcal{P}$ such that there exists a commutative diagram in $\mathcal A$ with exact rows.
\[
\xymatrix@R=2em@C=2.5em{
0 \ar[r]&X\ar[r]^-j \ar@{=}[d]
&P^0 \ar[r]^-f\ar[d]
&P^1 \ar[r]^-g\ar[d]
&P^2 \ar[r]^-{\overline h}\ar[d]
&\mathbb SX \ar[r]\ar[d]^-{c_X}
&0
\\
0 \ar[r]&X \ar[r]
&I_X\ar[r]
&I_{\Omega^{-1}X} \ar[r]
&I_{\Omega^{-2}X} \ar[r]
&\Omega^{-3}X \ar[r]
&0
}
\]
Here, the lower exact sequence is obtained by splicing the short exact sequence (\ref{ex:Omega})  for $X$, $\Omega^{-1}X$, and $\Omega^{-2}X$.
Moreover, we require that $(\mathbb Sj) \circ \overline{h}= h$, and that the morphism $c_X$ in $\mathcal{A}$ represents the morphism $\delta_X$ in $\underline{\mathcal A}$.

We will say that a triangle in $(\mathcal{P}, \mathbb{S}|_\mathcal{P})$ is \emph{$\delta$-exact}  if it belongs to $\Delta_\delta$. The class $\Delta_\delta$ does not depend on the choice of the representatives $c_X$ above. 

\vskip5pt

The following classification result is due to {\cite[Theorem~16.4]{Hel68}};  see also \cite[Proposition~3.2]{Mur20}.

\begin{theorem}
\label{thm:heller-classification}
Let $\mathcal A$ be a Frobenius abelian category with an autoequivalence $\mathbb S$. Then the assignment 
$$\delta\longmapsto \Delta_\delta$$
yields a bijection between the set of Heller parameters  on  $(\mathcal{A}; \mathbb{S})$ and the set of pretriangulated structures on $(\mathcal{P},\mathbb S|_\mathcal{P})$.
\end{theorem}

A Heller parameter $\delta$ on  $(\mathcal{A}; \mathbb{S})$  is called \emph{triangulated},  if the resulting pretriangulated category $(\mathcal{P},\mathbb S|_\mathcal{P}, \Delta_\delta)$ is triangulated. In general, it seems to be highly nontrivial to give a criterion on triangulated Heller parameters.

Denote by ${\rm Ab}$ the category of abelian groups. For an additive category $\mathcal{C}$, we denote by ${\rm mod}(\mathcal{C})$ the category of coherent functors on $\mathcal{C}$. Here, 
we recall that a contravariant additive functor $F\colon \mathcal{C}\rightarrow {\rm Ab}$ is said to be \emph{coherent} if there is an exact sequence of functors
$${\rm Hom}_\mathcal{C}(-, C_1) \longrightarrow {\rm Hom}_\mathcal{C}(-, C_0) \longrightarrow F \longrightarrow 0$$
for some $C_i\in \mathcal{C}$. The Yoneda embedding
$$H_\mathcal{C}\colon \mathcal{C}\longrightarrow {\rm mod}(\mathcal{C}), \; C \mapsto {\rm Hom}_\mathcal{C}(-, C)$$
allows us to view $\mathcal{C}$ as a full subcategory of  ${\rm mod}(\mathcal{C})$.

The following well-known example shows that most pretriangulated categories arise in the manner of Theorem~\ref{thm:heller-classification}. 

\begin{example}\label{exm:coherent}
    Let $(\mathcal{T}, \Sigma, \Delta)$ be a pretriangulated category, which has split idempotents. Set $\mathcal{A}={\rm mod}(\mathcal{T})$. By \cite[Theorem~3.1]{Freyd} and \cite[Théorème~3.2.1]{Ver96}, $\mathcal{A}$ is a Frobenius  abelian category. The Yoneda embedding $H_\mathcal{T}$ identifies $\mathcal{T}$ with the full subcategory $\mathcal{P}$ of projective-injective functors. The suspension functor $\Sigma$ on $\mathcal{T}$ extends uniquely to an autoequivalence $\mathbb{S}$ on $\mathcal{A}$. In particular, we identify $(\mathcal{T}, \Sigma)$ with $(\mathcal{P}, \mathbb{S}|_\mathcal{P})$. 
    
    Theorem~\ref{thm:heller-classification} implies that there is a unique Heller parameter $\delta$ on $(\mathcal{A}; \mathbb{S})$ such that $\Delta$ corresponds to $\Delta_\delta$. Consequently,  the given pretriangulated category $(\mathcal{T}, \Sigma, \Delta)$  is triangulated if and only if the Heller parameter $\delta$ is triangulated. 
\end{example}

The following consideration is due to \cite[Section~8]{Ami07}. Let $\mathcal A$ be a Frobenius abelian category with a fixed autoequivalence $\mathbb S$.  Assume that there is an exact sequence consisting of exact endofunctors on $\mathcal{A}$
\begin{equation}
\label{eq:strict-exact-sequence}
0\longrightarrow \Id
\longrightarrow \mathbb{P}^0
\longrightarrow \mathbb{P}^1
\longrightarrow \mathbb{P}^2
\longrightarrow \mathbb S
\longrightarrow 0,
\end{equation}
with $\mathbb{P}^j(X)\in\mathcal{P}$ for every $X\in\mathcal A$ and $j=0,1,2$. This sequence of functors yields the following commutative diagram in $\mathcal{A}$.
\[
\xymatrix@R=2em@C=2.5em{
0 \ar[r]&X \ar[r] \ar@{=}[d]
&\mathbb{P}^0(X) \ar[r] \ar[d]
&\mathbb{P}^1(X) \ar[r] \ar[d]
&\mathbb{P}^2(X) \ar[r] \ar[d]
&\mathbb SX \ar[r]\ar[d]^-{\delta_X}
&0
\\
0 \ar[r]&X \ar[r]
&I_X\ar[r]
&I_{\Omega^{-1}X} \ar[r]
&I_{\Omega^{-2}X} \ar[r]
&\Omega^{-3}X \ar[r]
&0
}
\]
The image of $\delta_X$ in $\underline{\mathcal{A}}$ is unique and an isomorphism. Moreover, these morphisms $\delta_X$ give rise to a Heller parameter $\delta$ on $(\mathcal{A}; \mathbb{S})$. Such a Heller parameter will be called \emph{strict}.

The following remarkable result is due to {\cite[Theorem~8.1]{Ami07}}.

\begin{theorem} \label{th:strict-Heller-parameter}
Any strict Heller parameter on $(\mathcal{A}; \mathbb{S})$ is triangulated. 
\end{theorem}

We mention that the converse of the statement above does not hold, in general.

\begin{example}
Consider  $\mathcal{A}={\rm mod}(\mathbb{Z}/{4\mathbb{Z}})$,   the category of finitely generated modules over $\mathbb{Z}/{4\mathbb{Z}}$; it is a Frobenius abelian category. The full subcategory $\mathcal{P}$ of projective-injective objects coincides with the category of free $\mathbb{Z}/{4\mathbb{Z}}$-modules of finite rank. There is a unique natural isomorphism $\delta\colon {\rm Id}\rightarrow \Omega^{-3}$; moreover, it is a Heller parameter on $(\mathcal{A}; {\rm Id})$. By \cite[Theorem~1]{Muro_2007}, this Heller parameter $\delta$ is triangulated. 

We claim that $\delta$ is not strict. Assume that there is an exact sequence (\ref{eq:strict-exact-sequence}) of functors with $\mathbb{S}={\rm Id}$ giving rise to $\delta$. We observe that all the functors $\mathbb{P}^i$ vanish on the simple module $\mathbb{F}_2$. This implies that the evaluation of the sequence  (\ref{eq:strict-exact-sequence}) on $\mathbb{F}_2$ is not exact, which is a contradiction. 
\end{example}

We mention the above example coming from \cite{Muro_2007} is of algebraic nature but without differential graded or topological enhancements; compare \cite{rizzardo2018} and \cite{CM26, HL26}. Also, compare this with \cite[Theorem (1)]{Mur22}.

 Recall that $Z_\vartriangle(\underline{\mathcal{A}})^\times$ is the unit group of the triangle center of $\underline{\mathcal{A}}$. The following observation is immediate; see {\cite[Corollary~16.5]{Hel68}}.

\begin{lemma}\label{lem:Z-action}
 Let $\delta$ be a Heller parameter on $(\mathcal{A}; \mathbb S)$ and $u\in Z_\vartriangle(\underline{\mathcal{A}})^\times$. Then $\delta\circ(\mathbb S u)$ is also  a Heller parameter on $(\mathcal{A}; \mathbb S)$.
\end{lemma}

The lemma above yields a free and transitive $Z_\vartriangle(\underline{\mathcal{A}})^\times$-action on the set of Heller parameters on $(\mathcal{A}; \mathbb S)$. The key fact is that this action does not preserve triangulated Heller parameters, in general.  We will use this central action to perturb Heller parameters, from a triangulated one to a non-triangulated one.

\section{An obstruction on lightning flashes}\label{sec:lightning-flash}
In this section, we obtain an obstruction on lightning flashes, which arises from two triangulated Heller parameters; see Proposition~\ref{prop:lightning-flash-obstruction}.

Let $\mathcal{A}$ be a Frobenius abelian category with a fixed autoequivalence $\mathbb{S}$. Denote by $\mathcal{P}$ the full subcategory of projective-injective objects.

Denote by $\mathbf K_{\mathrm{ac}}(\mathcal{P})$ the homotopy category of acyclic complexes with projective components.  It is well known that taking the zeroth cocycles yields a triangle equivalence
\begin{align}\label{equ:Z^0}
    Z^0\colon\mathbf{K}_{\mathrm{ac}}(\mathcal{P})
\xlongrightarrow{\ \sim\ }
\underline{\mathcal A};
~~
P^\bullet\longmapsto Z^0(P^\bullet).
\end{align}
Here, we recall that $Z^0(P^\bullet)$ is given by the kernel of $d_P^0\colon P^0\rightarrow P^1$; see \cite[Proposition~4.4.18]{Kra21}.

Following \cite[Subsection 8.1]{Ami07}, an \emph{$\mathbb S$-complex} is a complex $P^\bullet$ in $\mathcal{P}$ such that
\[
P^{n+3}=\mathbb S P^n ~\text{ and } \quad d_P^{n+3}=\mathbb S (d_P^n)
\]
for each $n\in\mathbb Z.$  For a cochain map $f^\bullet \colon P^\bullet \rightarrow Q^\bullet$ between two $\mathbb S$-complexes, we have another cochain map
$$f^\bullet_\mathbb{S}\colon P^\bullet \longrightarrow Q^\bullet$$
given by $f_\mathbb{S}^n=\mathbb{S}(f^{n-3})\colon P^n=\mathbb{S}P^{n-3}\rightarrow \mathbb{S}Q^{n-3}=Q^n$. When $f^\bullet=f_\mathbb{S}^\bullet$ as cochain maps, we will say that $f^\bullet$ is an $\mathbb{S}$-cochain map.

In what follows, we assume that $\delta$ is a Heller parameter on $(\mathcal{A}; \mathbb{S})$. Fix $u\in Z_\vartriangle(\underline{\mathcal{A}})^\times$ and set $\widetilde\delta=\delta\circ (\mathbb{S}u)$; see Lemma~\ref{lem:Z-action}.

\begin{lemma} \label{lemma:a^Sigma}
 Suppose that $P^\bullet$ is an $\mathbb S$-complex whose truncation $P^0\rightarrow P^1\rightarrow P^2\rightarrow \mathbb S P^0$  is $\delta$-exact, and that $Q^\bullet$ is an $\mathbb S$-complex whose truncation $Q^0\rightarrow Q^1\rightarrow Q^2\rightarrow \mathbb S Q^0$ is $\widetilde\delta$-exact. Then any cochain map $f^\bullet \colon P^\bullet\rightarrow Q^\bullet$ necessarily satisfies
\[
Z^0(f^\bullet_\mathbb S)= Z^0(f^\bullet)\circ u_{Z^0(P^\bullet)}
\]
in $\underline{\mathcal A}$. In particular, if $f^\bullet$ is an isomorphism in $\mathbf K_{\mathrm{ac}}(\mathcal{P})$, then we have 
$$u_{Z^0(P^\bullet)}= Z^0((f^\bullet)^{-1}\circ f^\bullet_\mathbb S).$$
\end{lemma}
\begin{proof}
Put $M= Z^0(P^\bullet)$ and $N= Z^0(Q^\bullet)$. The cochain map $f^\bullet_\mathbb S\colon P^\bullet\rightarrow Q^\bullet$ yields the following  commutative diagram in $\mathcal A$.
\[
\xymatrix@R=2em@C=2.2em{
0 \ar[r]&M \ar[r]\ar[d]_-{Z^0(f^\bullet_\mathbb S)}
&P^0 \ar[r]\ar[d]_-{f_\mathbb S^0}
&P^1 \ar[r]\ar[d]_-{f_\mathbb S^1}
&P^2 \ar[r]\ar[d]_-{f_\mathbb S^2}
&\mathbb SM \ar[r]\ar[d]^-{\mathbb{S}Z^0(f^\bullet)}
&0
\\
0 \ar[r]&N \ar[r]
&Q^0 \ar[r]
&Q^1 \ar[r]
&Q^2 \ar[r]
&\mathbb SN \ar[r]
&0
}
\]
The rightmost vertical map is $\mathbb S Z^0(f^\bullet)$ because it is induced by
$f_\mathbb S^3=\mathbb S(f^0)$.

Recall the construction of $\delta$-exact and $\widetilde{\delta}$-exact triangles in \Cref{Hel-para}. By the comparison lemma, there is a diagram in $\mathcal A$. 

\[
\xymatrix@C=.6em@R=1.2em{
0 \ar[rr]&&M \ar[rr]\ar[dr]_-{Z^0(f^\bullet_\mathbb S)\;\; \; \; \; }  \ar@{=}[dd]
&&P^0 \ar[rr]\ar[dr]^-{f_\mathbb S^0}\ar@{..>}[dd]
&&P^1 \ar[rr]\ar[dr]^-{f_\mathbb S^1}\ar@{..>}[dd]
&&P^2 \ar[rr]\ar[dr]^-{f_\mathbb S^2}\ar@{..>}[dd]
&&\mathbb SM \ar[rr]\ar[dr]^-{\mathbb SZ^0(f^\bullet)}  \ar[dd]_(.6){c_M}
&&0&{}
\\
{}&0 \ar[rr]&&N \ar[rr]\ar@{=}[dd]
&&Q^0 \ar[rr]\ar@{..>}[dd]
&&Q^1 \ar[rr]\ar@{..>}[dd]
&&Q^2 \ar[rr]\ar@{..>}[dd]
&&\mathbb SN \ar[rr]\ar[dd]^(.6){c_N}
&&0
\\
0 \ar[rr]&&M \ar[rr]\ar[dr]_-{Z^0(f_\mathbb S^\bullet)}
&&I_M \ar[rr]\ar@{..>}[dr]
&&I_{\Omega^{-1}M} \ar[rr]\ar@{..>}[dr]
&&I_{\Omega^{-2}M} \ar[rr]\ar@{..>}[dr]
&&\Omega^{-3}M \ar[rr]\ar[dr]_-{\Omega^{-3}Z^0(f^\bullet_\mathbb S)}
&&0&{}
\\
{}&0 \ar[rr]&&N \ar[rr]
&&I_N \ar[rr]
&&I_{\Omega^{-1}N} \ar[rr]
&&I_{\Omega^{-2}N}\ar[rr]
&&\Omega^{-3}N \ar[rr]
&&0
}
\]
Here, $c_M$ represents $\delta_M$,  and $c_N$ represents $\widetilde{\delta}_N$ in $\underline{\mathcal A}$. All the squares on the boundary  commute, whereas the interior vertical squares do not necessarily commute. We observe that  the rightmost vertical square commutes in the stable category $\underline{\mathcal A}$, that is, we have
\[
\Omega^{-3}Z^0(f^\bullet_\mathbb S) \circ \delta_M
=\widetilde\delta_N\circ \mathbb SZ^0(f^\bullet) \text{ in }\underline{\mathcal A}.
\] 

Therefore, we have 
\begin{align*}
    \Omega^{-3}Z^0(f_\mathbb S^\bullet)\circ \delta_M& =\widetilde\delta_N\circ \mathbb S Z^0(f^\bullet)\\
    & =\delta_N\circ \mathbb S(u_N)\circ \mathbb SZ^0(f^\bullet)\\
    & =\delta_N\circ \mathbb S(u_N\circ Z^0(f^\bullet)) \\
    & =\Omega^{-3}(u_N\circ Z^0(f^\bullet))\circ \delta_M.
\end{align*}
The last equality uses the naturality of $\delta$. 

Since $\delta_M$ is an isomorphism and $\Omega^{-3}$ is an autoequivalence on $\underline{\mathcal{A}}$, we infer from the identity above that 
$$Z^0(f^\bullet_\mathbb S)=u_N\circ Z^0(f^\bullet)$$
holds in $\underline{\mathcal A}$. By the naturality of $u$, we have $u_N\circ Z^0(f^\bullet)=Z^0(f^\bullet)\circ u_M$. We conclude $Z^0(f^\bullet_\mathbb S)=Z^0(f^\bullet)\circ u_M$, and complete the proof.
\end{proof}

Set $\eta=u-\Id$. Then we have $\widetilde \delta=\delta\circ \mathbb{S}(\Id+\eta)$. The following result provides an obstruction for certain Heller parameters to be triangulated.

\begin{proposition} \label{prop:lightning-flash-obstruction}
Keep the notation and assumptions above. Assume further that both Heller parameters $\delta$ and $\widetilde{\delta}$ on $(\mathcal{A}; \mathbb{S})$ are triangulated. Then for an exact triangle
\[
\xi\colon U \xlongrightarrow{w} Y \xlongrightarrow{z} X
\longrightarrow \Omega^{-1}U
\]
in $\underline{\mathcal A}$ with
$\eta_U=0=\eta_Y$, the endomorphism $\eta_X\colon  X\rightarrow X$ is a lightning flash on $\xi$. 
\end{proposition}

\begin{proof}
Recall the construction of $\delta$-exact and $\widetilde{\delta}$-exact triangles in \Cref{Hel-para}. Take  $P^\bullet$  to be an $\mathbb S$-complex whose truncation $\varepsilon\colon P^0\rightarrow P^1\rightarrow P^2\rightarrow \mathbb S P^0$ is $\delta$-exact with $Z^0(P^\bullet)=U$. Take  $Q^\bullet$ to be an $\mathbb{S}$-complex whose truncation 
 $\varepsilon'\colon Q^0\rightarrow Q^1\rightarrow Q^2\rightarrow \mathbb S Q^0$ is $\delta$-exact with $Z^0(Q^\bullet)=Y$. 
 It follows from the construction of $\delta$-exact triangles that  both $\mathbb{S}$-complexes $P^\bullet$ and $Q^\bullet$ are acyclic.

 Since $\eta_U=0$ and $\eta_Y=0$, we have $\delta_U=\widetilde{\delta}_U$ and $\delta_Y=\widetilde{\delta}_Y$. It follows that 
 both triangles $\varepsilon$ and $\varepsilon'$ are also $\widetilde\delta$-exact.

Choose morphisms $f^0\colon P^0\rightarrow Q^0$ and $f^1\colon P^1\rightarrow Q^1$ such that the following 
diagram
\[
\xymatrix@R=1.5em@C=3em{
0 \ar[r]&U \ar[r]\ar[d]_-w
&P^0 \ar[r]\ar[d]_-{f^0}
&P^1 \ar[d]^-{f^1}
\\
0 \ar[r]&Y \ar[r]
&Q^0 \ar[r]
&Q^1
}
\]
commutes. We apply  (TR3) and \Cref{thm:good-fill-in-criterion} to the resulting triangulated category $(\mathcal{P}, \mathbb{S}|_\mathcal{P}, \Delta_\delta)$. 
We infer that there exists a morphism $f^2\colon P^2\rightarrow Q^2$ such that 
$$(f^0, f^1, f^2) \colon \varepsilon\longrightarrow \varepsilon'$$ 
is a morphism between two $\delta$-exact triangles whose truncated cone $\operatorname{C}(f^0, f^1, f^2)$ is $\delta$-exact. 

We repeat the argument above by replacing $\delta$ with $\widetilde \delta$.  We infer that  there exists a morphism $\widetilde f^2 \colon P^2\rightarrow Q^2$ such that 
$$(f^0, f^1, \widetilde f^2)\colon \varepsilon\longrightarrow \varepsilon'$$
is a morphism between the two $\widetilde\delta$-exact triangles whose truncated cone $\operatorname{C}(f^0, f^1, \widetilde f^2)$ is still $\widetilde\delta$-exact.

There is  a unique $\mathbb{S}$-cochain map  $f^\bullet\colon P^\bullet\rightarrow Q^\bullet$  extending the triple $(f^0, f^1, f^2)$. Similarly, there is a unique $\mathbb{S}$-cochain map $\widetilde f^\bullet\colon P^\bullet\rightarrow Q^\bullet$ extending the triple $(f^0, f^1, \widetilde f^2)$. We observe that $Z^0(f^\bullet)=w=Z^0(\widetilde f^\bullet)$. By the triangle equivalence (\ref{equ:Z^0}), we infer that the cochain maps  $f^\bullet$  and $\widetilde f^\bullet$ are homotopic.

We assume that $f^\bullet$ is homotopic to $\widetilde f^\bullet$ via a homotopy $H=(h^n\colon P^{n+1}\rightarrow Q^n)_{n\in \mathbb{Z}}$. We observe that $f^\bullet$ is homotopic to $\widetilde f^\bullet$ via another homotopy $H_\mathbb S=(\mathbb S (h^{n-3})\colon P^{n+1}\rightarrow Q^{n})_{n\in \mathbb{Z}}$. By \Cref{factsoncomplex}, there is a lightning flash on the standard triangle.
\[
\xymatrix@R=2em@C=3em{
{\rm stan}\colon P^\bullet \ar[r]^-{f^\bullet}
&Q^\bullet \ar[r]^-{\binom01}
&\Cone(f^\bullet) 
\ar[d]^-{(\theta^\bullet_H)^{-1}\circ \theta^\bullet_{H_\mathbb{S}} - \Id}  
\ar[r]^-{(1,0)}
&\Sigma P^\bullet
\\
{\rm stan}\colon  P^\bullet \ar[r]_-{f^\bullet}
&Q^\bullet \ar[r]_-{\binom01}
&\Cone(f^\bullet) \ar[r]_-{(1,0)}
&\Sigma P^\bullet
}
\]
Applying the triangle equivalence (\ref{equ:Z^0}) to the diagram above,   we obtain the following lightning flash on $\xi'=Z^0({\rm stan})$.
\[
\xymatrix@R=2em@C=3em{
\xi'\colon U \ar[r]^-w
&Y \ar[r]^-{z'}
&X' \ar[d]^-{Z^0((\theta^\bullet_H)^{-1}\circ \theta^\bullet_{H_\mathbb{S}} -\Id)}
\ar[r]
&\Omega^{-1}U
\\
\xi' \colon U \ar[r]^-w
&Y \ar[r]^-{z'}
&X' \ar[r]
&\Omega^{-1}U
}
\]
Here, we have $X'=Z^0({\rm Cone}(f^\bullet))$. 

Consider the cochain map $\theta_H^\bullet \colon \Cone(f^\bullet) \rightarrow \Cone(\widetilde{f}^\bullet)$. We observe that  both mapping cones $\Cone(f^\bullet)$ and $\Cone(\widetilde f^\bullet)$ are $\mathbb S$-complexes. Moreover, their truncations are given by the truncated cones $\operatorname{C}(f^0,f^1,f^2)$ and $\operatorname{C}(f^0,f^1,\widetilde f^2)$, which are $\delta$-exact and $\widetilde{\delta}$-exact, respectively. In view of (\ref{equ:theta_H}), we have $\theta^\bullet_{H_\mathbb{S}}=(\theta_H^\bullet)_\mathbb{S}$. This allows us to apply \Cref{lemma:a^Sigma} and infer the following identity.
$$Z^0((\theta^\bullet_H)^{-1}\circ \theta^\bullet_{H_\mathbb{S}} -\Id)=Z^0((\theta^\bullet_H)^{-1}\circ (\theta_H^\bullet)_{\mathbb{S}}) -\Id_{X'}=u_{X'}-\Id_{X'}=\eta_{X'}$$ 
This implies that $\eta_{X'}\colon X'\rightarrow X'$ is a lightning flash on $\xi'$.

Finally, there is an isomorphism $c\colon X\rightarrow X'$ inducing an isomorphism between  the two exact triangles $\xi$ and $\xi'$. It follows that $c^{-1}\circ \eta_{X'}\circ c\colon X\rightarrow X$ is a lightning flash on $\xi$. By the naturality of $\eta$, we have 
$c^{-1}\circ \eta_{X'}\circ c=\eta_X$. This completes the proof. 
\end{proof}

\section{Gemstones and lightning flashes}\label{sec:gemstone}

In this section, we introduce  gemstones in a pretriangulated category. Such objects give rise to endomorphisms which  are not lightning flashes; see Proposition~\ref{prop:gem3}. 

We fix a ground field $k$. Throughout this section, we assume that $(\TT,\Sigma, \Delta)$ is a Hom-finite $k$-linear pretriangulated category which has split idempotents. It follows that $\mathcal{T}$ is a Krull-Schmidt category; see \cite[Appendix~A]{CYZ}. A split monomorphism will be called a section, and a split epimorphism will be called a retraction. 

An exact triangle $X\xrightarrow{f}Y\xrightarrow{g}Z \xrightarrow{h} \Sigma X$ is said to be an \emph{almost-split triangle} if both $X$ and $Z$ are indecomposable, $h\ne 0$, and any non-retraction $t\colon M\rightarrow Z$ factors through $g$. In this case, by \cite[Lemma I.4.2]{Hap88} any non-section $s\colon X\rightarrow M$ factors through $f$. For more details, we refer readers to \cite[Section I.4]{Hap88}. 

\begin{remark}
    The first half in the original proof of \cite[Lemma I.4.2]{Hap88} uses (TR4), which is superfluous. Indeed, we only use the middle row in the octahedron  in \cite[p.32]{Hap88}, which can be deduced by  (TR3) and a rotation. In summary,  \cite[Lemma~I.4.2]{Hap88} holds for a pretriangulated category. 
\end{remark}

Following \cite[Definition~2.1]{Kun06}, a nonzero morphism $u\colon  Y\rightarrow Z$ is called \emph{almost-vanishing} if $s\circ u = 0$ and $u\circ t = 0$ for any non-section $s \colon Z\rightarrow W$ and any non-retraction $t\colon  W\rightarrow Y$; see also \cite[Subsection~2.2]{CY18}. A morphism $u\colon Y\rightarrow Z$ is almost-vanishing if and only if it fits into an almost-split triangle $\Sigma^{-1}Z \xrightarrow{} X \xrightarrow{} Y \xrightarrow{u} Z$.

The following fact is well known. 

\begin{lemma}\label{almostvanishing}
    Let $u\colon Y \rightarrow Z$ be an almost-vanishing morphism. Then it factors through any nonzero morphism $g\colon C\rightarrow Z$ and factors through any nonzero morphism $g'\colon Y\rightarrow C$.
\end{lemma}

\begin{proof}
   By duality, we only prove the first half of the statement.   We complete $g$ into an exact triangle $C\xrightarrow{g}Z\xrightarrow{h} D\rightarrow\Sigma C$.  Since $g\ne 0$, the morphism $h$ is not a section; see \cite[Lemma~I.1.4]{Hap88}. Therefore we have $h\circ u=0$. It follows that $u$ factors through $g$. 
\end{proof}

The following new concept is central in this work.

\begin{definition}\label{defn:gemstone}
 An indecomposable object $X\in \TT$ is called a \emph{gemstone} provided that it satisfies the following conditions. 
\begin{enumerate}
\item[\textup{(G1)}] 
There is an isomorphism $\gamma \colon  X\stackrel{\sim}{\rightarrow} \Sigma^2X$, and $\Hom_\TT(X, \Sigma X)=0$. 
\item[\textup{(G2)}] There is an almost-vanishing morphism $\alpha\colon X\rightarrow  X$, which is a non-isomorphism and belongs to the center of $\End_\TT(X)$. Furthermore, we require that $\gamma \circ \alpha=\Sigma^2\alpha\circ \gamma$.
\item[\textup{(G3)}] Consider the almost-split triangle associated with $\Sigma \alpha$
\begin{equation}\label{tri:T1}
X\xlongrightarrow{u}U\xlongrightarrow{v}\Sigma X \xlongrightarrow{\Sigma \alpha} \Sigma X. 
\end{equation}
We require that $U$ is not isomorphic to $\Sigma U$. 
\end{enumerate}
\end{definition}

We remind the reader that the compatibility condition in (G2) is nontrivial.

\begin{remark}\label{rem:gemstone}
Assume that $X$ is a gemstone in $\mathcal{T}$. Then the following elementary facts hold. 

    (1) By \textup{(G1)}, we have $\Hom_\TT(\Sigma X, X)=0$.

    (2) By the isomorphism $\gamma$ and \cite[Proposition I.4.3]{Hap88}, we infer that $U$ is isomorphic to $\Sigma^2 U$. 

    (3) Since $\alpha$ is  a non-section and almost-vanishing, we infer that $\alpha^2=0$.
\end{remark}

By Remark~\ref{rem:gemstone}(3), we have $\Sigma \alpha\circ \Sigma \alpha=0$. By the exact triangle (\ref{tri:T1}), we infer that $\Sigma \alpha$ factors through $v$. Take a morphism $s\colon \Sigma X\rightarrow U$ satisfying 
\begin{align}\label{equ:s}
    \Sigma \alpha=v\circ s,
\end{align} 
and  complete $s$ into an exact triangle 
\begin{equation}\label{tri:T2}
\Sigma X \xlongrightarrow{s} U\xlongrightarrow{w}Y\xlongrightarrow{z} \Sigma ^2X .
\end{equation}
We observe that the morphism $s$ is nonzero. 

In what follows, we will first prove that neither $X$ nor $\Sigma X$ is a direct summand of $U$ or $Y$.  We denote by $\dim $ the dimension of some  vector space  over $k$. 

\begin{lemma}\label{gem1}
Assume that $X$ is a gemstone. Keep the notation above. Then the following facts hold.
\begin{enumerate}
    \item[\textup{(1)}]  $\dim \Hom_\TT(\Sigma X, U)=\dim \Hom_\TT(X, U)$.

    \item[\textup{(2)}]  $\dim \End_\TT(\Sigma X)>\dim \Hom_\TT(\Sigma X, U) >\dim \Hom_\TT(\Sigma X, Y)$. 

\item[\textup{(3)}]  $\dim \End_\mathcal{T}(X) > \dim \Hom_\TT(X, U).$

\item[\textup{(4)}]  $\dim \Hom_\mathcal{T}(\Sigma X, \Sigma U) > \dim \Hom_\TT(\Sigma X, Y).$
\end{enumerate}
Consequently, neither $X$ nor $\Sigma X$ is a direct summand of $U$. Moreover, neither $\Sigma X$ nor $\Sigma U$ is  a direct summand of $Y$. 
\end{lemma}

\begin{proof}
(1) Applying $\Hom_\TT(\Sigma X, -)$ to (\ref{tri:T1})  and using Remark~\ref{rem:gemstone}(1), we get an exact sequence
\begin{align}\label{seq:v-mono}
0\longrightarrow \Hom_\TT(\Sigma X, U)\xlongrightarrow{\Hom_\TT(\Sigma X, v)}\End_\TT(\Sigma X)\xlongrightarrow{\Hom_\TT(\Sigma X, \Sigma \alpha )} \End_\TT(\Sigma X).    
\end{align}
Applying $\Hom_\TT(X, -)$ to (\ref{tri:T1}) and using \textup{(G1)}, we obtain an exact sequence
\[
\End_\TT(X)\xlongrightarrow{\Hom_\TT(X, -\alpha)}\End_\TT(X)\xlongrightarrow{\Hom_\TT(X, u)} \Hom_\TT(X, U)\longrightarrow 0. 
\]
Then we have 
\[
\begin{aligned}
    \dim \Hom_\TT(\Sigma X, U)&=\dim \End_\TT(\Sigma X)-\dim \Im \Hom_\TT(\Sigma X, \Sigma \alpha )\\
    &=\dim \End_\TT(X)-\dim \Im \Hom_\TT(X, -\alpha )\\
    &=\dim \Hom_\TT(X, U).
\end{aligned}
\]

(2)  By \textup{(G2)}, we have $\Sigma \alpha\neq 0$. It follows that $\Hom_\mathcal T(\Sigma X,\Sigma \alpha)$ is nonzero.  By the exact sequence (\ref{seq:v-mono}), we infer that   $\dim \End_\TT(\Sigma X)>\dim \Hom_\TT(\Sigma X, U)$. 

Applying $\Hom_\TT(\Sigma X, -)$ to (\ref{tri:T2}) and using \textup{(G1)}, we obtain  an exact sequence
\[
\End_\TT(\Sigma X)\xlongrightarrow{\Hom_\TT(\Sigma X, s)}\Hom_\TT(\Sigma X, U)\xlongrightarrow{\Hom_\TT(\Sigma X, w)} \Hom_\TT(\Sigma X, Y)\longrightarrow 0.  
\]
The map $\Hom_{\mathcal T}(\Sigma X,s) $ is nonzero, since $s$ is nonzero. The exact sequence above implies that $
\dim \Hom_{\mathcal T}(\Sigma X,U)
>
\dim \Hom_{\mathcal T}(\Sigma X,Y).
$

(3) Observe that $\dim \End_\TT(X)=\dim \End_\TT(\Sigma X)$. Then the required inequality follows from (1) and (2). The proof of (4) is similar. 

The two consequences follow immediately from the corresponding inequalities  on dimensions of Hom spaces.  
\end{proof}

The morphism $s\colon \Sigma X\rightarrow U$ above enjoys a nice property.  

\begin{lemma}\label{gem1.5}
We have  $g\circ s=0$ for any morphism $g\colon U\rightarrow \Sigma U$ in $\mathcal{T}$.
\end{lemma}

\begin{proof}

We first observe that any morphism $g\colon U\rightarrow \Sigma U$ is a non-section. Otherwise, $g$ has to be an isomorphism, since $U$ and $\Sigma U$ have the same number of indecomposable direct summands. But this contradicts \textup{(G3)}.

   We  complete $g$ into an exact triangle
\[
W\xlongrightarrow{t}U\xlongrightarrow{g} \Sigma U \longrightarrow\Sigma W.
\]
Recall the morphism $v$ in (\ref{tri:T1}). We claim that $v\circ t\ne 0$. 

Assume on the contrary that $v\circ t=0$. It follows that $v$ factors through $g$, say $v=q\circ g$ for some morphism $q\colon\Sigma U\rightarrow \Sigma X$. By \cite[Proposition~I.4.3]{Hap88}, $v$ is irreducible, and hence  $q$ is a retraction as $g$ is not a section. But it is impossible, since $\Sigma X$ is not a direct summand of $\Sigma U$ by Lemma~\ref{gem1}, proving the claim above. 

By \Cref{almostvanishing}, the almost-vanishing morphism $\Sigma \alpha$ factors through $v\circ t$. Therefore, we have 
$$v\circ s=\Sigma \alpha=v\circ t\circ b$$ 
for some morphism  $b\colon  \Sigma X\rightarrow W$. By (\ref{seq:v-mono}), the map ${\rm Hom}_\TT(\Sigma X, v)$ is injective. It follows from the identity above that $s=t\circ b$. Consequently, we have 
$$g\circ s=g\circ t\circ b=0.$$
This completes the proof.
\end{proof}

\begin{lemma}\label{summandY}
   Keep the notation above. Then $X $ is not a direct summand of $Y$.  
\end{lemma}

\begin{proof}
Recall the isomorphism $\Sigma^2 X\cong X$.  It suffices to prove that any morphism $p\colon Y\rightarrow \Sigma^2 X$ is a non-retraction. 

By Lemma~\ref{gem1},  $\Sigma^2 X\cong X$ is not a direct summand of $U$. Therefore, $p\circ w\colon U\rightarrow \Sigma^2 X$ is not a retraction. By the almost-split triangle (\ref{tri:T1}),  $p\circ w$ factors through $\Sigma v$. Consequently, there is  a morphism $h\colon  U\rightarrow \Sigma U$ such that the following leftmost square commutes.
\begin{equation}\label{hp}
\xymatrix@R=2em@C=3em{
U \ar[r]^-w\ar[d]_-{h}
&Y \ar[r]^-z\ar[d]^-p
&\Sigma^2 X \ar@{.>}[d]^-{\ell} \ar[r]^{-\Sigma s}
&\Sigma U \ar[d]^-{\Sigma h}
\\
\Sigma U \ar[r]^-{\Sigma v}
&\Sigma^2 X\ar[r]^-{\Sigma^2 \alpha}
&\Sigma^2 X\ar[r]^-{\Sigma^2 u}
&\Sigma^2 U.
}
\end{equation}
By (TR3), there is a morphism $\ell\colon \Sigma^2 X \rightarrow  \Sigma^2 X$ making the diagram commute.

By Lemma~\ref{gem1.5}, we have $h\circ s=0$. Furthermore, we have 
$$ (\Sigma^2 u)\circ \ell=(\Sigma h)\circ(-\Sigma s)=0.$$
Consequently,  $\ell$ factors through $\Sigma^2 \alpha$.  Assume that $j\colon  \Sigma^2 X\rightarrow \Sigma^2 X$ satisfies  $\ell= (\Sigma^2 \alpha)\circ j$. 

The composite morphism  $j\circ z$ is not a retraction. Otherwise, since $\Sigma^2 X$ is indecomposable, we infer that $j$ is an isomorphism and $z$ is a retraction. But by (\ref{tri:T2}) and $s\neq 0$, the morphism $z$ is a non-retraction; see \cite[Lemma~I.2.4]{Hap88}.

We have 
$$(\Sigma^2 \alpha) \circ p=\ell \circ z=(\Sigma^2 \alpha) \circ j \circ z=0. $$
Here, for the rightmost equality, we use the facts that $j\circ z$ is not a retraction and that $\Sigma^2 \alpha$ is almost-vanishing. The identity above implies that $p$ is a non-retraction. 
\end{proof}

The following result justifies the notion of a gemstone.

\begin{proposition}\label{prop:gem3}
 Assume  that $X$ is a gemstone in $\mathcal{T}$. Keep the notation above. Then any nonzero endomorphism $f \colon  \Sigma^2 X\rightarrow \Sigma^2X$ is not a lightning flash on  the exact triangle $U\xrightarrow{w}Y\xrightarrow{z} \Sigma^2 X \xrightarrow{-\Sigma s} \Sigma U$. 
\end{proposition}

The exact triangle above is rotated from (\ref{tri:T2}). 

\begin{proof}
Suppose on the contrary that there is a morphism $\varphi\colon  \Sigma U\rightarrow Y$ with $f=z\circ \varphi \circ (-\Sigma s)$: 
\[
\xymatrix@C=2.5em@R=1.2em{
U \ar[r]^-w
&Y \ar[r]^-z
&\Sigma^2X \ar[dd]_(.35)f \ar[r]^-{-\Sigma s}
&\Sigma U\\
& & & \\
U \ar[r]^-w
&Y \ar[r]^-z
&\Sigma^2X \ar[r]^-{-\Sigma s}
&\Sigma U.
\save
\ar@{.}
  "1,4";
  "2,4"+<0pt,10pt>
\ar@{.}@/^5pt/
  "2,4"+<0pt,10pt>;
  "2,4"+<-10pt,0pt>
\ar@{.}
  "2,4"+<-10pt,0pt>;
  "2,2"+<10pt,0pt>
  _(.25){\varphi}
\ar@{.}@/_5pt/
  "2,2"+<10pt,0pt>;
  "2,2"+<0pt,-10pt>
\ar@{.>}
  "2,2"+<0pt,-10pt>;
  "3,2"
\restore
}
\]
Applying $\Hom_\TT(-, \Sigma^2X)$ to (\ref{tri:T1}) and using \textup{(G1)}, we obtain an exact sequence
\[
\End_\TT(\Sigma^2X)\xlongrightarrow{\Hom_\TT(\Sigma^2 \alpha, \Sigma^2X)}\End_\TT(\Sigma^2X)\xlongrightarrow{\Hom_\TT(-\Sigma v, \Sigma ^2 X)}\Hom_\TT(\Sigma U, \Sigma^2X)\longrightarrow 0. 
\]
There exists a morphism $\epsilon\in \End_\TT(\Sigma^2X)$ satisfying $z\circ\varphi=\epsilon\circ(-\Sigma v)$.  We have
$$f=z\circ \varphi \circ (-\Sigma s)=\epsilon\circ(-\Sigma v)\circ (-\Sigma s)=\epsilon \circ \Sigma^2\alpha.$$
The rightmost equality uses (\ref{equ:s}). 

Since $\Sigma^2\alpha$ is almost-vanishing and $f\ne 0$, we infer that $\epsilon$ is a section, and thus is an automorphism. 

Therefore, we have the following factorization. 
$$-\Sigma v=\epsilon^{-1}\circ z\circ\varphi$$
By \Cref{summandY} and \textup{(G1)}, the morphism $\epsilon^{-1}\circ z\colon  Y\rightarrow \Sigma^2 X$ is a non-retraction. By \cite[Proposition~I.4.3]{Hap88}, $-\Sigma v$ is irreducible, and hence the factorization above implies that $\varphi\colon  \Sigma U\rightarrow Y$ is a section, which contradicts the consequences in \Cref{gem1}. 
\end{proof}

\section{The main results}\label{sec:main-results}

In this section, we prove our main results using Propositions~\ref{prop:lightning-flash-obstruction} and \ref{prop:gem3}. 

Let $k$ be a field. We assume that $\mathcal{A}$ is a Hom-finite $k$-linear Frobenius abelian category. Then the stable category $\underline{\mathcal{A}}$ is Hom-finite Krull-Schmidt; in particular, it has split idempotents. 

Suppose that there is a gemstone $X$ in $\underline{\mathcal{A}}$. Keep the notation in Definition~\ref{defn:gemstone}. By \cite[Lemma~2.7]{CY18}, there is a unique natural isomorphism  $u\colon  {\rm Id}_{\underline{\mathcal A}}\rightarrow {\rm Id}_{\underline{\mathcal A}}$ 
such that 
$$u_{X}={\rm Id}_X+\alpha, ~~~ u_{\Omega^{-1} X}={\rm Id}_{\Omega^{-1}X}+\Omega^{-1} \alpha, 
~~ \text{ and } ~~u_M={\rm Id}_M$$
for any indecomposable object $M$ which is not isomorphic to $X$ or $\Omega^{-1} X$. Here, we recall that $\Omega^{-1}$ is the suspension functor on $\underline{\mathcal{A}}$.

\begin{lemma}\label{lem:X-u}
    The natural isomorphism $u$ belongs to $Z_\vartriangle(\underline{\mathcal{A}})^\times.$
\end{lemma}

\begin{proof}
    It suffices to prove that $u\Omega^{-1}=\Omega^{-1}u$.  In view of \cite[Lemma~2.7]{CY18}, it suffices to prove $u_{\Omega^{-1}X}=\Omega^{-1}(u_X)$ and $u_{\Omega^{-2}X}=\Omega^{-1}(u_{\Omega^{-1}X})$.
    
    The first identity is trivial. For the second one, we use the isomorphism $\gamma \colon X\rightarrow \Omega^{-2}X$ in Definition~\ref{defn:gemstone}\textup{(G1)}. By the naturality of $u$, we have  
    $$u_{\Omega^{-2}X}=\gamma \circ u_X\circ \gamma^{-1}={\rm Id}_{\Omega^{-2}X}+\Omega^{-2}\alpha=\Omega^{-1}(u_{\Omega^{-1}X}).$$
  Here, the middle equality uses the identity in Definition~\ref{defn:gemstone}\textup{(G2)}. This completes the proof.
  \end{proof}

We fix a $k$-linear autoequivalence $\mathbb{S}$ on $\mathcal{A}$, and fix a Heller parameter $\delta$ on $(\mathcal{A}; \mathbb{S})$. Set $\widetilde{\delta}=\delta\circ (\mathbb{S}u)$; see Lemma~\ref{lem:Z-action}.

\begin{theorem}\label{mainthm}
    Keep the notation above. Assume that the given Heller parameter $\delta$ is triangulated and that $X$ is a gemstone in $\underline{\mathcal{A}}$. Then the Heller parameter $\widetilde{\delta}$ is not triangulated.   
\end{theorem}

\begin{proof}
Set $\eta=u-\Id$. By the proof of Lemma~\ref{lem:X-u}, we infer that $\eta_{\Omega^{-2}X}=\Omega^{-2}\alpha$. This morphism is nonzero, since $\alpha$ is nonzero. 

 Associated to the gemstone $X$, we consider the exact triangle (\ref{tri:T2}) in $\underline{\mathcal{A}}$. By a rotation, we have the following exact triangle
\[
\xi\colon U\xlongrightarrow{w}Y\xlongrightarrow{z}\Omega^{-2}X\xlongrightarrow{-\Omega^{-1} s} \Omega^{-1}U
\]
in $\underline{\mathcal A}$. By Lemmas~\ref{gem1} and \ref{summandY}, neither $X$ nor $\Omega^{-1}X$ is a direct summand of $U$ or $Y$.  It follows from the construction of $u$ that $\eta_U=0=\eta_Y$.

Assume on the contrary that $\widetilde{\delta}$ is triangulated. By Proposition~\ref{prop:lightning-flash-obstruction}, the endomorphism
$\eta_{\Omega^{-2}X}=\Omega^{-2}\alpha$ is a lightning flash on $\xi$. As mentioned before, this endomorphism is nonzero. Therefore, we obtain the  desired contradiction to Proposition~\ref{prop:gem3}. 
\end{proof}

\begin{remark}\label{rem:mainthm}
  Recall that $\delta$ is a fixed triangulated Heller parameter on $(\mathcal{A}; \mathbb{S})$. The perturbed Heller parameter $\widetilde{\delta}=\delta^{(X, \alpha)}$  depends on the pair $(X, \alpha)$. For each nonzero scalar $\lambda$, we might replace $\alpha$ by $\lambda \alpha$ and obtain a non-triangulated Heller parameter $\delta^{(X, \lambda \alpha)}$. 
  
  When the ground field $k$ is infinite, we actually obtain infinitely many non-triangulated Heller parameters on $(\mathcal{A}; \mathbb{S})$. Consequently, by Theorem~\ref{thm:heller-classification} we obtain infinitely many non-triangulated pretriangulated categories 
  $$(\mathcal{P}, \mathbb{S}|_\mathcal{P}, \Delta^{(X, \lambda \alpha )})$$ 
  with the pretriangulated structure $\Delta^{(X, \lambda \alpha)}$ corresponding to  $\delta^{(X, \lambda \alpha)}$.  
\end{remark}

In view of Example~\ref{exm:coherent}, we have an immediate variant of Theorem~\ref{mainthm}. 

\begin{theorem}\label{thm:anotherform}
Let $(\mathcal{T}, \Sigma, \Delta)$ be a Hom-finite $k$-linear triangulated category with split idempotents. Assume that there is a gemstone in $\underline{{\rm mod}}(\mathcal{T})$. Then there is a non-triangulated pretriangulated category $(\mathcal{T}, \Sigma, \widetilde{\Delta})$.     
\end{theorem}

\begin{proof}
We keep the notation in Example~\ref{exm:coherent}.   The triangulated structure $\Delta$ on $(\mathcal{T}, \Sigma)$ corresponds to a triangulated Heller parameter $\delta$ on $({\rm mod}(\mathcal{T}); \mathbb{S})$. Applying Theorem~\ref{mainthm}, we have the perturbed one $\widetilde{\delta}$ which is non-triangulated. The non-triangulated pretriangulated structure $\Delta_{\widetilde{\delta}}$  on $(\mathcal{P}, \mathbb{S}|_\mathcal{P})$ corresponds to the required one $\widetilde{\Delta}$ on $(\mathcal{T}, \Sigma)$.  
\end{proof}

\section{Examples of gemstones} \label{sec:example}

The main result of this section, \Cref{example-X}, provides a construction of gemstones. This allows us to apply Theorem~\ref{thm:anotherform} to concrete examples, and confirm the conjecture in Introduction.

Let $k$ be a field and $\mathcal{T}$ be a Hom-finite $k$-linear triangulated category. By a \emph{Serre functor} on $\mathcal{T}$, we mean a $k$-linear autoequivalence $\Phi\colon \mathcal{T}\rightarrow \mathcal{T}$ together with a bifunctorial isomorphism
$$\Hom_\TT(X, Y)\simeq D\Hom_\TT (Y, \Phi X).$$
Here, $D=\Hom_k(-, k)$ denotes the duality. The Serre functor $\Phi$ becomes naturally a triangle functor; see \cite[Proposition~3.3]{BK90} and compare \cite[Theorem~A.4.4]{Boc08}. For an integer $d$, the triangulated category $\mathcal{T}$ is said to be \emph{$d$-Calabi-Yau} provided that it has a Serre functor $\Phi$, and it is isomorphic to $\Sigma^d$ as triangle functors; see \cite[Subsection 2.6]{Kel08} and \cite[Subsection 5.1]{GKO13}.

Let $\mathcal A$ be a Hom-finite $k$-linear Frobenius abelian category. Denote by $\mathcal P$ the full subcategory of projective-injective objects. Recall that $\Omega^{-1}$ is the suspension functor on the stable category $\underline{\mathcal{A}}$. 


Fix a $k$-linear autoequivalence $\mathbb{S}$ on $\mathcal{A}$. Assume that $\delta$ is a triangulated Heller parameter on $(\mathcal{A}; \mathbb{S})$. We have the resulting triangulated category  $(\mathcal P, \mathbb S|_\mathcal P, \Delta_\delta)$.

\begin{proposition}\label{example-X}
  Suppose that $(\mathcal P,\mathbb S|_\mathcal P,  \Delta_\delta)$ is   $d$-Calabi-Yau for some odd integer $d$. Assume that there exist two indecomposable objects $M, L\in \mathcal P$ and nonzero morphisms $q\colon M\rightarrow L$, $\iota\colon  L\rightarrow M$ satisfying the following conditions.
    \begin{enumerate}
        \item[\textup{(1)}] The following sequence
        \begin{equation}\label{LM}
        \cdots\xlongrightarrow{}L
\xlongrightarrow{\iota}M
\xlongrightarrow{q}L
\xlongrightarrow{\iota}M
\xlongrightarrow{}\cdots
        \end{equation}
        is exact in $\mathcal{A}$.
        \item[\textup{(2)}] $\Hom_\mathcal A(M, \iota)=0$, and $\Hom_\mathcal A(L, q)=0$.
        \item[\textup{(3)}] The endomorphism algebra $\End_\mathcal A(M)$ is commutative but not a field.
        \item[\textup{(4)}] For each $a\in \End_\mathcal A(M)$, there is a central element $c\in \End_\mathcal A(M\oplus L)$ such that $c|_M=a$. 
    \end{enumerate}
Then $X\coloneqq\Coker(\iota)$ is a gemstone in $\underline{\mathcal A}$. 
\end{proposition}

Here $c|_M$ means the restriction of $c$ on the direct summand $M$, whose image is also in $M$. Note that a central element $c\in \End_\mathcal A(M\oplus L)$ must be diagonal. 

In the following proof, we will use almost-split sequences in $\mathcal{A}$. Recall that such sequences induce almost-split triangles in  $\underline{\mathcal{A}}$.  We refer to \cite[Chapter~V]{ARS95} for details.  

\begin{proof}
The sequence \eqref{LM} gives a minimal projective resolution and a minimal injective resolution for  $X$.  We will identify $\Omega^{-1}X$ with the cokernel of $q$. Since both $L$ and $M$ are indecomposable projective, it follows that both $X$ and $\Omega^{-1} X$ are indecomposable and non-projective. 

\vskip 3pt

\emph{Step 1.}\; The sequence \eqref{LM} allows us to identify  $\Omega^{-2}X$ with $X$ as objects. 
We observe from (2) that
\begin{equation}\label{two-null}
    \Hom_\mathcal A(M,\Omega^{-1}X)=0\text{ and }\Hom_\mathcal A(L,X)=0.
\end{equation}
It follows immediately that $\Hom_{\mathcal A}(X,\Omega^{-1} X)=0=\Hom_{\mathcal A}(\Omega^{-1}X,X)$. This implies (G1). 

\vskip 3pt

\emph{Step 2.}\;  We claim that the endomorphism algebra $\End_{\underline {\mathcal A}}(X)$ is isomorphic to $\End_\mathcal A(M)$. Consequently, by (3) it is commutative and not a field.

By $\Hom_{\mathcal A}(\Omega^{-1}X,X)=0$, we observe that each morphism in $\End_\mathcal A(M)$ induces a unique morphism in $\End_\mathcal A(X)$. This yields an algebra isomorphism
$$ \End_\mathcal A(M)\simeq \End_\mathcal A(X)$$
by using \eqref{two-null}. Since any morphism $\beta\in\End_\mathcal A(X)$ factoring through a projective-injective object necessarily factors through the embedding $X\to L$, we get $\beta=0$ by $\Hom_\mathcal A(L, X)=0$ in \eqref{two-null}. It follows that $\End_{\underline{\mathcal{A}}}(X)=\End_\mathcal A(X)$, proving the claim.

Next, we show the existence of almost-vanishing morphism $\alpha\in \End_\mathcal A(X)$. Recall that $\mathbb{S}^d|_\mathcal{P}$ is the Serre functor on $\mathcal{P}$. We infer  a bifunctorial isomorphism 
$${\rm Hom}_\mathcal{A}(P, W)\simeq D{\rm Hom}_\mathcal{A}(W, \mathbb{S}^d P)$$
for $P\in \mathcal{P}$ and $W\in \mathcal{A}$. By \cite[Lemma~5.3]{GKO13}, $\underline{\mathcal A}$ admits a Serre functor, which is isomorphic to $\Omega \mathbb{S}^d\simeq \Omega^{1-3d}$. Here, we recall from Section~\ref{Hel-para} the isomorphism $\mathbb{S}\simeq \Omega^{-3}$. 
 It follows from \cite[Propositions~I.2.1,~2.3 and Theorem~I.2.4]{ReitenVanDenBergh2002} that there is an almost-vanishing morphism $$\alpha \in \Hom_{\underline{\mathcal A}}(X, \Omega^{1-3d} X)= \End_{\underline{\mathcal A}}(X);$$ 
 see also \cite[Proposition~2.8]{Che11}. Here, we use the assumption that $d$ is odd.

The claim above and (3) imply that there exists a nonzero endomorphism $f$ on $X$ which is a non-retraction. Since $\alpha \colon X\rightarrow X$ is almost-vanishing, we have $\alpha \circ f=0$. Consequently, $\alpha$ is not an isomorphism.

Finally, we consider the action of $\Omega^{-2}$ on $\End_\mathcal A(X)$. For each $g\in \End_\mathcal A(X)$, there is $a\in \End_\mathcal A(M)$ such that $g$ is induced by $a$. By (4), there is a central element $c\in \End_\mathcal A(M\oplus L)$ such that $c|_M=a$. Set $b=c|_L\in \End_\mathcal A(L)$, and regard the sequence \eqref{LM} as a complex $P^\bullet$, whose even components are $M$ and odd components are $L$. Since $c$ is central in $\End_\mathcal A(M\oplus L)$, there is a cochain map $c^\bullet\colon P^\bullet\to P^\bullet$, where $c^{2n}=a$ and $c^{2n+1}=b$. Consequently, $\Omega^{-2}g=g$ under the identification $\Omega^{-2}X=X$. This finishes the proof of (G2).


\vskip 3pt

\emph{Step 3.}\; The almost-vanishing morphism $\Omega^{-1} \alpha\colon \Omega^{-1} X\to \Omega^{-1} X$ in $\underline{\mathcal{A}}$ induces an almost-split sequence in $\mathcal A$
\begin{equation}\label{seq:7.1}
0\longrightarrow X \xlongrightarrow{u} U \xlongrightarrow{v} \Omega^{-1} X \longrightarrow 0. 
\end{equation}
It yields the following almost-split triangle in $\underline{\mathcal{A}}$.
$$X\xlongrightarrow{u}U\xlongrightarrow{v}\Omega^{-1} X \xlongrightarrow{\Omega^{-1} \alpha} \Omega^{-1} X$$

Similarly, the almost-vanishing morphism $\alpha\colon X\to X$ induces an almost-split sequence in $\mathcal A$ 
\begin{equation}\label{seq:7.2}
0\longrightarrow \Omega^{-1}X\xlongrightarrow{u'} U'\xlongrightarrow{v'} X \longrightarrow 0,
\end{equation}
which yields an almost-split triangle in $\underline{\mathcal A}$.
\begin{equation*}
\Omega^{-1}X\xlongrightarrow{{u'}}U'\xlongrightarrow{{v'}}X \xlongrightarrow{ \alpha} X
\end{equation*}
By comparing the two almost-split triangles above, we have an isomorphism   $U'\cong \Omega^{-1}U$ in $\underline{\mathcal A}$.

We claim that $U$ is not isomorphic to $U'$ in $\underline{\mathcal A}$. This implies (G3) and completes the whole proof. We will prove the claim by showing that $U$ and $U'$ are indecomposable, non-projective and not isomorphic in $\mathcal A$. 

First, by $\Hom_\mathcal A(X, \Omega^{-1}X)=0$ and \cite[Lemma 6.5]{HKR08}, we obtain that $U$ is indecomposable in $\mathcal A$. Similarly, $U'$ is also indecomposable in $\mathcal A$. Since neither $\alpha$ nor $\Omega^{-1}\alpha$ is an isomorphism in $\underline{\mathcal{A}}$, we infer from the triangles above that both $U$ and $U'$ are nonzero in $\underline{\mathcal{A}}$. That is, both $U$ and $U'$ are non-projective in $\mathcal{A}$.

Applying $\Hom_\mathcal A(X, -)$ to \eqref{seq:7.1}, we get  $\Hom_\mathcal A(X, X)\cong \Hom_\mathcal A(X, U)$. Also, applying $\Hom_\mathcal A(X, -)$ to \eqref{seq:7.2}, we get an embedding $\Hom_\mathcal A(X, U')\rightarrow \Hom_\mathcal A(X, X)$, which is not surjective. Thus by comparing the dimension of the Hom spaces, we infer that $U$ is not isomorphic to $U'$ in $\mathcal A$.  Combining these results, we infer the claim. 
\end{proof}

We have the following concrete example for Proposition~\ref{example-X}.

\begin{example}\label{exm:longexact}
Let $R$ be a finite-dimensional commutative local symmetric $k$-algebra, and $A=R[x]/(x^{3m})$ with $m\geq 2$. 
Then $A$ is a symmetric algebra. Hence, the stable module category $\smodcat(A)$ is triangulated with the suspension $\Omega_A^{-1}$. Moreover, it is $(-1)$-Calabi-Yau; see \cite[Corollary 6.12]{BM24}.

 We fix two indecomposable $A$-modules $L_0=A/(x^m)$ and $M_0=A/(x^{2m})$. Consider two morphisms 
\[
    \iota_0\colon L_0\longrightarrow M_0 ~( \overline{r}\longmapsto \overline{x^mr})\text{ and } q_0\colon M_0\longrightarrow L_0~(\overline r\longmapsto \overline r).
\]
We have $\Omega_A^{-1}(L_0)=M_0$. Moreover, there is an exact triangle
\begin{align}\label{tri:LML}
L_0\xlongrightarrow{\iota_0}M_0\xlongrightarrow{q_0}L_0\xlongrightarrow{-\iota_0}M_0    
\end{align}
in $\smodcat(A)$. We observe that $\underline{\rm Hom}_A(M_0, \iota_0)=0=\underline{\rm Hom}_A(L_0, q_0)$. Note that $\underline{\End}_A(M_0)\cong R[x]/(x^{m})$ is commutative, and it is not a field as $m\geq 2$.

Set $\mathcal{A}={\rm mod}(\smodcat(A))$. This is a Frobenius abelian category. The Yoneda embedding allows us to identify $\smodcat(A)$ with $\mathcal{P}$, the full subcategory of projective-injective functors in $\mathcal{A}$.
By \Cref{exm:coherent}, we have an autoequivalence $\mathbb{S}$ on $\mathcal{A}$ extending $\Omega_A^{-1}$ on $\smodcat(A)$. Moreover, there is a triangulated Heller parameter $\delta$ on $(\mathcal{A}; \mathbb{S})$ such that $(\mathcal{P}, \mathbb{S}|_\mathcal{P}, \Delta_\delta)$ is identified with $\smodcat(A)$ as triangulated categories.

Set $L=\Hom_{\mathcal P}(-,L_0)$ and $M=\Hom_{\mathcal P}(-,M_0)$. Set $\iota=\Hom_{\mathcal P}(-,\iota_0)$ and $q=\Hom_{\mathcal P}(-,q_0)$. 
The exact triangle (\ref{tri:LML}) yields a long exact sequence as in \Cref{example-X}(1). The remaining conditions in  \Cref{example-X} are fulfilled. We point out that here the commutativity of $A$ plays a key role in (3) and (4). Consequently,  $\Coker(\iota)$ is a gemstone in $\underline{\mathcal A}$.

By Theorem~\ref{thm:anotherform}, we have a non-triangulated pretriangulated structure $\widetilde{\Delta}$ on $(\smodcat(A), \Omega_A^{-1})$. This proves Corollary~\ref{application}. 

When $R=k$,  $\smodcat(A)$ is equivalent to the category ${\rm proj}(\Pi (A_{3m-1}))$ of finitely generated projective modules over the preprojective algebra $\Pi (A_{3m-1})$ of type $A_{3m-1}$; see \cite[Section~7]{DR92}. Consequently, we obtain a non-triangulated pretriangulated structure on ${\rm proj}(\Pi (A_{3m-1}))$. This result is also obtained independently in \cite{HL26}.

\end{example}

\subsection*{Declaration of AI Tools}
This research began with a very concrete example involving the preprojective algebra $\Pi(A_5)$ over $\mathbb{F}_2$, where the authors received assistance from AI tools. We are responsible for all the results and proofs in this paper. 

\subsection*{Acknowledgments} 
We are very grateful to Amnon Neeman for bringing together the two groups of authors, one consisting of the second and fifth authors and the other of the remaining authors, after realizing that they were independently working on closely related projects. 

We thank Osamu Iyama,  Bernhard Keller and Fernando Muro for their interest and help. The first author thanks Ping He for the reference \cite{HL26}, which appeared about one month after our first arXiv version.  We are deeply indebted to Alexander Generalov  and Pu Zhang for their continued guidance and encouragement. 

We thank the JIUCHONG team at USTC for providing access to Eureka, and thank Jie Li and Tianyang Sun for their assistance.

\end{document}